\documentclass[12pt]{article}
\usepackage[utf8]{inputenc}
\usepackage[T1]{fontenc}
\usepackage{amsmath, amssymb, amsthm}
\usepackage{hyperref}

\title{A unified parametric approach to the Erdős–Straus conjecture with explicit solutions for a set of integers of natural density one.}

\author{Philemon Urbain Mballa \\[1ex]
\href{mailto:philemon-urbain.mballa@etu.u-paris.fr}{philemon-urbain.mballa@etu.u-paris.fr} \\
\href{mailto:philemonmballa@gmail.com}{philemonmballa@gmail.com}
}

\date{}

\newtheorem{proposition}{Proposition}

\begin{document}
\maketitle

\begin{abstract}
We develop a parametric approach to study the Diophantine equation 
$\frac{k}{n} = \frac{1}{x} + \frac{1}{y} + \frac{1}{z}$, underlying the 
Erdős--Straus ($k=4$), Sierpiński ($k=5$), and related generalizations. 
We introduce and analyze the properties of the fundamental function 
$F_{x,t}^{(k)}(n) = t^2(kx-n)^2 - 2nxt$, 
whose being a perfect square is equivalent to yielding a solution of these conjectures.

In the classical Erdős--Straus case ($k=4$), for the residue classes 
$n \equiv 0,2,3 \pmod{4}$, we provide explicit symmetric solutions $y=z$, 
covering already 75\% of all integers. 
For the historically most resistant class $n \equiv 1 \pmod{4}$, 
we construct explicit symmetric solutions based on the existence of a divisor 
$b \equiv 3 \pmod{4}$, and we further show that this condition 
is satisfied for almost all such integers: the set of exceptions has 
natural density zero. Consequently, the Erdős--Straus conjecture is verified 
for a proportion of integers tending to $1$ in this class.

These results yield infinitely many new families of explicit solutions not covered 
by previous constructions, highlight the structural behavior of $F$.
\end{abstract}

\section{Introduction}

The Erdős--Straus conjecture, formulated in 1948 by Paul Erdős and Ernst G. Strauss, states that for every integer \(n\ge 2\), there exist positive integers \(x,y,z\) such that
\[
\frac{4}{n} = \frac{1}{x} + \frac{1}{y} + \frac{1}{z}. \tag{1}
\]
Despite numerous works, this conjecture remains open to this day. Table 1, adapted from Tao's article \cite{2}, traces the history of numerical verifications that have progressively pushed forward the bound up to which the conjecture is validated.
\begin{table}[htbp]
\centering
\begin{tabular}{|c|c|c|}
\hline
\textbf{Bound} & \textbf{Year} & \textbf{Author(s)} \\
\hline
\(5000\) & \(\leqslant 1950\) & Straus, see \cite{3} \\
\(8000\) & 1962 & Bernstein \cite{4} \\
\(20000\) & \(\leqslant 1969\) & Shapiro, see \cite{5} \\
\(106128\) & 1948/9 & Oblath \cite{6} \\
\(141648\) & 1954 & Rosati \cite{7} \\
\(10^{7}\) & 1964 & Yamomoto \cite{8} \\
\(1.1\times 10^{7}\) & 1976 & Jollensten \cite{9} \\
\(10^{8}\) & 1971 & Terzi \cite{10} \\
\(10^{9}\) & 1994 & Elsholtz \& Roth (unpublished) \\
\(10^{10}\) & 1995 & Elsholtz \& Roth (unpublished) \\
\(1.6\times 10^{11}\) & 1996 & Elsholtz \& Roth (unpublished) \\
\(10^{10}\) & 1999 & Kotsireas \cite{11} \\
\(10^{14}\) & 1999 & Swett \cite{12} \\
\(2\times 10^{14}\) & 2012 & Bello-Hernández, Benito, Fernández \cite{13} \\
\(10^{17}\) & 2014 & Salez \cite{14} \\
\(10^{18}\) & 2025 & Mihnea \& Bogdan \cite{15} \\
\hline
\end{tabular}
\caption{Numerical verifications of the Erdős--Straus conjecture.}
\end{table}

The Polish mathematician Wacław Sierpiński generalized this question by replacing the numerator \(4\) with \(5\):
\[
\frac{5}{n} = \frac{1}{x} + \frac{1}{y} + \frac{1}{z}. \tag{2}
\]
This conjecture, known as the Sierpiński conjecture, has also been the subject of numerous investigations. Article \cite{16} presents a detailed history of numerical verifications for \(k=5\).

A natural generalization, often attributed to Andrzej Schinzel (a student of Sierpiński), consists in considering for any fixed integer \(k\ge 4\) the equation
\[
\frac{k}{n} = \frac{1}{x} + \frac{1}{y} + \frac{1}{z}. \tag{3}
\] for all natural numbers \(n\), except for finitely many \(n\).
In the literature, almost all works focus on the cases \(k=4\) and \(k=5\), which is naturally explained by the increasing complexity of the problem as \(k\) grows.

In our previous article \cite{1}, we proposed a unified approach to all these conjectures. For any fixed integer \(k\ge 4\), we introduced the function
\[
F_{x,t}^{(k)}(n) = t^2(kx-n)^2 - 2nxt
\]
and established the following fundamental equivalence:
\[
\frac{k}{n} = \frac{1}{x} + \frac{1}{y} + \frac{1}{z} \;\Longleftrightarrow\; \exists x, t\in\mathbb{N}^*,\; F_{x,t}^{(k)}(n) \text{ is a perfect square } m^2, m\in\mathbb{N},
\] for every integer \(n\ge N_1\) where \(N_1\) is an integer greater than or equal to \(2\),
with then \(y = t(kx-n) + m\) and \(z = t(kx-n) - m\). This equivalence reduces the search for solutions to the condition that \(F_{x,t}^{(k)}(n)\) is a perfect square.

The present article deepens this approach by studying the fundamental properties of the function \(F\). For each fixed pair \((x,t)\), we define its admissible domain \(\mathcal{D}_{x,t}^{(k)}\) and prove that on this domain, \(F\) is strictly decreasing, nonnegative, and converges to its minimum. A key result, the **Zero Lemma**, establishes that if \(F(n_0)=0\) for some \(n_0\) in the domain, then \(n_0\) is necessarily the upper bound of \(\mathcal{D}_{x,t}^{(k)}\).

These properties, although local (specific to each pair), are essential for understanding the behavior of \(F\) and constitute a fundamental step toward characterizing the pairs \((x,t)\) for which \(F\) is a perfect square.

 \section*{Decrease, boundedness and convergence on the admissible domain}

\subsection*{Admissible domain of a pair}

Let \(k \ge 4\) be a fixed integer and for a given pair \((x,t) \in \mathbb{N}^{*2}\), we define its \textbf{admissible domain}:
\[
\mathcal{D}_{x,t}^{(k)} = \left\{ n \in \mathbb{N},\ n \ge N_1\ge 2 \ \middle|\ n < kx \ \text{and}\ t \ge \frac{2nx}{(kx-n)^2} \right\}.
\]
The condition \(n < kx\) ensures that \(x > n/k\) (necessary condition in the quadratic equivalence). The second condition guarantees the positivity of \(F\).

\begin{proposition}[Decrease and positivity on the domain]
Let \(k \ge 4\) and \((x,t) \in \mathbb{N}^{*2}\) be fixed. For every \(n \in \mathcal{D}_{x,t}^{(k)}\), we have:
\[
F_{x,t}^{(k)}(n) = t^2(kx-n)^2 - 2nxt \ge 0.
\]
Moreover, for all \(n_1, n_2 \in \mathcal{D}_{x,t}^{(k)}\) with \(n_1 > n_2\), we have:
\[
F_{x,t}^{(k)}(n_1) < F_{x,t}^{(k)}(n_2).
\]
Thus, \(n \mapsto F_{x,t}^{(k)}(n)\) is strictly decreasing on \(\mathcal{D}_{x,t}^{(k)}\).
\end{proposition}

\begin{proof}
Let \(n_1,n_2 \in \mathcal{D}_{x,t}^{(k)}\) with \(n_1 > n_2\).

Since \(n_2 < kx\) (because \(n_2 \in \mathcal{D}_{x,t}^{(k)}\)), we have \(kx - n_2 > 0\). Similarly, \(kx - n_1 > 0\) and \(kx - n_1 < kx - n_2\).

The function \(y \mapsto y^2\) being strictly increasing on \(\mathbb{R}_+\), we obtain:
\[
(kx - n_1)^2 < (kx - n_2)^2.
\]
Multiplying by \(t^2 > 0\):
\[
t^2(kx - n_1)^2 < t^2(kx - n_2)^2.
\]

Furthermore, \(n_1 > n_2\) and \(x,t > 0\) give:
\[
-2n_1xt < -2n_2xt.
\] 
Adding these two inequalities, we obtain \(F(n_1) < F(n_2)\), which establishes the strict decrease .

\vspace{0.5em}
\textbf{Positivity} follows directly from the condition \(t \ge \frac{2nx}{(kx-n)^2}\), which is equi\-valent to : 
\[
t^2(kx-n)^2 \ge 2nxt.
\]
since \(t>0\), which subsequently gives \(F \ge 0\).
\end{proof}

\begin{proposition}[Convergence and boundedness on the domain]
Let \(k \ge 4\) be a fixed integer and let \((x,t) \in \mathbb{N}^{*2}\) be a given pair. The sequence \(n \mapsto F_{x,t}^{(k)}(n)\), defined for \(n \in \mathcal{D}_{x,t}^{(k)}\), satisfies:

\begin{enumerate}
    \item It is strictly decreasing.
    \item It is bounded below by \(0\).
    \item By the monotone convergence theorem applied to the finite sequence \((F(n))_{n \in \mathcal{D}_{x,t}^{(k)}}\), it attains its infimum at the last element of the domain. Let us denote
\[
L_{x,t} = \min_{n \in \mathcal{D}_{x,t}^{(k)}} F_{x,t}^{(k)}(n) \ge 0.
\]
In particular, \(L_{x,t} = F_{x,t}^{(k)}(N_{x,t})\) where \(N_{x,t} = \max \mathcal{D}_{x,t}^{(k)}\).
    \item For every \(n \in \mathcal{D}_{x,t}^{(k)}\), we have the inequality:
    \[
    0 \le L_{x,t} \le F_{x,t}^{(k)}(n) \le F_{x,t}^{(k)}(m_{x,t}),
    \]
    where \(m_{x,t} = \min \mathcal{D}_{x,t}^{(k)}\) (the smallest admissible value).
    \item In particular, \(F_{x,t}^{(k)}\) is bounded on \(\mathcal{D}_{x,t}^{(k)}\) and we may take 
    \[
    C(x,t) = F_{x,t}^{(k)}(m_{x,t})
    \]
    as an explicit upper bound (independent of \(n\) within this domain).
\end{enumerate}
\end{proposition}

\begin{proof}
Points 1 and 2 follow from the previous proposition.  
Point 3 is a direct application of the monotone convergence theorem.  
Point 4 uses the decreasing property: for all \(n \ge m_{x,t}\) in the domain, \(F(n) \le F(m_{x,t})\). Moreover, since \(L_{x,t}\) is the limit, we have \(L_{x,t} \le F(n)\) for all \(n\) (as the sequence decreases to its limit).  
Point 5 is an immediate consequence of point 4.
\end{proof}

\section*{Quadratic equivalence theorem}

Let \(k \ge 4\) be a fixed integer, and let \(n \ge N_1\ge 2\) be a given integer.  
Let \(x \in \mathbb{N}^*\) be such that 
\[
x \ge \left\lfloor \frac{n}{k} \right\rfloor + 1 \quad (\text{which ensures } kx > n).
\]

Then the following two statements are equivalent:

\begin{enumerate}
    \item There exist \(y, z \in \mathbb{N}^*\) such that 
    \[
    \frac{k}{n} = \frac{1}{x} + \frac{1}{y} + \frac{1}{z}.
    \]
    
    \item There exists \(t \in \mathbb{N}^*\) such that 
    \[
    F_{x,t}^{(k)}(n) = t^2(kx - n)^2 - 2nxt
    \]
    is a perfect square.
\end{enumerate}

Moreover, when (2) holds and we set 
\(m = \sqrt{F_{x,t}^{(k)}(n)} \in \mathbb{N}\), 
an explicit solution of (1) is given by
\[
y = t(kx - n) + m, \qquad z = t(kx - n) - m,
\]
(or the reverse order).

\begin{proof}
\textbf{(1) $\Rightarrow$ (2).} 
Suppose there exist \(y,z \in \mathbb{N}^*\) satisfying (1). Then
\[
\frac{1}{y} + \frac{1}{z} = \frac{k}{n} - \frac{1}{x} = \frac{kx - n}{nx}.
\]
Reducing to a common denominator yields
\[
\frac{y + z}{yz} = \frac{kx - n}{nx},
\]
that is,
\begin{equation}
nx(y + z) = yz(kx - n). \tag{*}
\end{equation}

Set 
\[
S = y + z \quad \text{and} \quad P = yz.
\]
Equation (*) becomes \(nxS = P(kx - n)\). 
Since \(S\) and \(P\) are integers, we may write \(S = 2t(kx - n)\) 
for some \(t \in \mathbb{N}^*\) (the factor \(2\) is introduced to simplify the discriminant computation). 
Substituting into (*), we obtain
\[
nx \cdot 2t(kx - n) = P(kx - n),
\]
hence \(P = 2nxt\). Thus,
\[
y + z = 2t(kx - n), \qquad yz = 2nxt.
\]

Consequently, \(y\) and \(z\) are the roots of the quadratic equation
\[
V^2 - 2t(kx - n)V + 2nxt = 0.
\]
Its discriminant is
\[
\Delta = \bigl[2t(kx - n)\bigr]^2 - 4 \cdot 2nxt 
= 4\bigl[t^2(kx - n)^2 - 2nxt\bigr] 
= 4\,F_{x,t}^{(k)}(n).
\]
For \(y\) and \(z\) to be integers, \(\Delta\) must be a perfect square, 
which is equivalent to \(F_{x,t}^{(k)}(n)\) being a perfect square.

\textbf{(2) $\Rightarrow$ (1).} 
Suppose there exists \( x,t \in \mathbb{N}^*\) such that \(F_{x,t}^{(k)}(n) = m^2\) 
with \(m \in \mathbb{N}\). Then the equation
\[
V^2 - 2t(kx - n)V + 2nxt = 0
\]
has discriminant \(\Delta = 4m^2\), so its roots are
\[
V = \frac{2t(kx - n) \pm 2m}{2} = t(kx - n) \pm m.
\]
Set
\[
y = t(kx - n) + m, \qquad z = t(kx - n) - m.
\]
A direct computation gives
\[
y + z = 2t(kx - n), \qquad 
yz = \bigl(t(kx - n) + m\bigr)\bigl(t(kx - n) - m\bigr) 
= t^2(kx - n)^2 - m^2.
\]
Since \(m^2 = t^2(kx - n)^2 - 2nxt\), we obtain \(yz = 2nxt\).
We then verify
\[
\frac{y + z}{yz} = \frac{2t(kx - n)}{2nxt} = \frac{kx - n}{nx},
\]
which implies
\[
\frac{1}{y} + \frac{1}{z} = \frac{k}{n} - \frac{1}{x},
\]
and finally \(\displaystyle \frac{k}{n} = \frac{1}{x} + \frac{1}{y} + \frac{1}{z}\).

The numbers \(y\) and \(z\) are indeed positive integers. It is immediate that \(y \in \mathbb{N}^*\), since \(kx - n > 0\) because
\[
x \ge \left\lfloor \frac{n}{k} \right\rfloor + 1 > \frac{n}{k}, 
\]
and with \(t \in \mathbb{N}^*\), \(x \in \mathbb{N}^*\), \(n \in \mathbb{N}^*\) and \(m \in \mathbb{N}\), we have \(y \in \mathbb{N}^*\).

Let us now show that \(z \in \mathbb{N}^*\). From
\[
m^2 = t^2(kx - n)^2 - 2nxt < t^2(kx - n)^2,
\]
since \(-2nxt < 0\) for \(n,x,t > 0\), and since the square root function is increasing on \(\mathbb{R}_+\), we obtain
\[
m < t(kx - n).
\]
Hence \(t(kx - n) - m > 0\), and since \(t,x,n,m\) are integers, it follows that
\[
z = t(kx - n) - m \in \mathbb{N}^*.
\]
\end{proof}

  \noindent\textbf{remark.} 
This equivalence holds for all \(n \ge N_1\), where \(N_1\) is a threshold depending on \(k\). The integer \(N_1\) is chosen such that for every \(n \ge N_1\), any solution \((x,y,z)\) of \(\frac{k}{n} = \frac{1}{x} + \frac{1}{y} + \frac{1}{z}\) yields an integer 
\[
t = \frac{y+z}{2(kx-n)} = \frac{yz}{2nx} \in \mathbb{N}^*,
\]
so that we may write \(y = t(kx-n) + m\) and \(z = t(kx-n) - m\) with \(m = \sqrt{F_{x,t}^{(k)}(n)} \in \mathbb{N}\).

For \(k = 4\) (Erdős--Straus), one has \(N_1 = 2\). For \(k \ge 5\), numerical experiments show the existence of a threshold \(N_0 < N_1\) below which solutions exist but do not yield an integer \(t\); instead they give a rational \(t\) that still satisfies the parametrization. This phenomenon is due to the scarcity of Egyptian fraction decompositions for small \(n\) in these cases. In such situations, \(F\) becomes a rational perfect square, yet still produces integer solutions \(y,z\) satisfying the conjecture.

In this work, we impose the condition \(t \in \mathbb{N}^*\) to ensure that \(F\) is an integer, which facilitates its study. The determination of an explicit value for \(N_1\) in terms of \(k\) remains an open problem, related to the complexity of the generalized Erdős--Straus and Sierpiński conjectures. Even in the current literature, the exact value of \(N_0\) (the smallest integer from which the decomposition exists) is not known as a function of \(k\); the generalized conjecture merely states that the decomposition exists for all but finitely many \(n\). Consequently, expressing \(N_1 \ge N_0\) explicitly is even more challenging.

Numerical simulations indicate that the gap between \(N_0\) and \(N_1\) grows slowly with \(k\). For \(k=4\), we have \(N_0 = N_1 = 2\). For \(k=5\) (Sierpiński's conjecture), \(N_0 = 2\) and \(N_1 = 11\). This does not pose a problem for our approach: if we can show that \(F\) always yields an integer perfect square for all \(n \ge N_1\), the remaining range \(N_0 \le n < N_1\) can be handled numerically using the same function \(F\) with rational values of \(t\). In  \cite{1}, we proved the elementary fact that \(N_0\) cannot be strictly less than \(k/3\)

\medskip

This last theorem reduces the search for solutions of the Erdős--Straus conjecture (\(k=4\)), the Sierpiński conjecture (\(k=5\)), or their generalizations (\(k \ge 6\)), to the problem of finding parameters \(x,t\) for which \(F_{x,t}^{(k)}(n)\) is a perfect square. 
The study of the boundedness, decrease, and other properties of \(F\) thus becomes central.

\noindent\textbf{lemma} (Zero Lemma) 
Let \(k \ge 4\) be an integer, and let \(x,t \in \mathbb{N}^*\) be fixed. Consider \(n \mapsto F_{x,t}^{(k)}(n)\) defined on its admissible domain 
\(\mathcal{D}_{x,t} = \{ n \ge N_1\ge 2 \mid n < kx \text{ and } t \ge \dfrac{2nx}{(kx-n)^2} \}\).

If \(F_{x,t}^{(k)}(n_0) = 0\) for some \(n_0 \in \mathcal{D}_{x,t}\), then \(n_0\) is necessarily the upper bound of \(\mathcal{D}_{x,t}\).

\begin{proof}
By Proposition 1, the sequence \(n \mapsto F_{x,t}^{(k)}(n)\) is strictly decreasing on \(\mathcal{D}_{x,t}\).

Suppose there exists \(n_0 \in \mathcal{D}_{x,t}\) with \(F(n_0)=0\), and that \(n_0\) is not the upper bound of \(\mathcal{D}_{x,t}\). Then there would exist \(n_1 > n_0\) with \(n_1 \in \mathcal{D}_{x,t}\). 

By strict decrease, we would have \(F(n_1) < F(n_0) = 0\). 
But since \(n_1 \in \mathcal{D}_{x,t}\), the positivity condition (Proposition 1) implies \(F(n_1) \ge 0\). Hence we obtain \(0 \le F(n_1) < 0\), a contradiction.

Therefore, such an \(n_1\) cannot exist, and \(n_0\) is indeed the greatest element of \(\mathcal{D}_{x,t}\).
\end{proof}

\begin{proposition}[Explicit symmetric solutions for three residue classes]
For the classical Erdős--Straus conjecture (\(k=4\)), all integers \(n \ge 2\) belonging to the residue classes
\[
n \equiv 0,\ 2,\ 3 \pmod{4}
\]
admit symmetric solutions \(y = z\). More precisely, for each such \(n\), there exists an explicit pair \((x,t) \in \mathbb{N}^{*2}\) such that \(F_{x,t}^{(4)}(n) = 0\), yielding
\[
y = z = t(4x-n) \quad \text{and} \quad \frac{4}{n} = \frac{1}{x} + \frac{1}{y} + \frac{1}{z}.
\]

\begin{proof}
For each residue class modulo \(4\) (except \(n \equiv 1 \pmod{4}\)), we construct explicit pairs \((x,t)\) with \(F_{x,t}^{(4)}(n)=0\):

\begin{itemize}
    \item If \(n = 4r\) with \(r \in \mathbb{N}^*\), take \(x = r+1\) and \(t = \dfrac{r(r+1)}{2}\). Then \(4x-n = 4\) and a direct computation shows \(F_{x,t}^{(4)}(n)=0\).
    
    \item If \(n = 4r+2\) with \(r \in \mathbb{N}\), take \(x = r+1\) and \(t = (2r+1)(r+1)\). Then \(4x-n = 2\) and \(F_{x,t}^{(4)}(n)=0\).
    
    \item If \(n = 4r+3\) with \(r \in \mathbb{N}\), take \(x = r+1\) and \(t = 2(4r+3)(r+1)\). Then \(4x-n = 1\) and \(F_{x,t}^{(4)}(n)=0\).
\end{itemize}

These three classes cover all integers except those congruent to \(1\) modulo \(4\), which represent asymptotically \(25\%\) of all integers. Hence at least \(75\%\) of integers are zeros of \(F\). For each such \(n\), the corresponding formulas yield symmetric solutions \(y = z = t(4x-n)\) that satisfy the conjecture.
\end{proof}
\end{proposition}

  \begin{proposition}[Explicit symmetric solutions for a subfamily of \(n \equiv 1 \pmod{4}\)]
Let \(n \equiv 1 \pmod{4}\), written as \(n = 4r + 1\) with \(r \in \mathbb{N}^*\). 
Let \(b\) be an odd integer such that \(b \equiv 3 \pmod{4}\), and define
\[
x = \frac{n+b}{4}.
\]
If \(b \mid n\) (equivalently \(b \mid (4r+1)\)), then there exists an integer \(t \in \mathbb{N}^*\) such that
\(F_{x,t}^{(4)}(n) = 0\). This yields the symmetric solution
\[
y = z = t(4x-n) = t b,
\]
satisfying
\[
\frac{4}{n} = \frac{1}{x} + \frac{1}{y} + \frac{1}{z}.
\]
\end{proposition}
\begin{proof}
Since $n = 4r+1$ and $b \equiv 3 \pmod{4}$, we have
\[
n+b \equiv 1+3 \equiv 0 \pmod{4},
\]
hence
\[
x = \frac{n+b}{4} \in \mathbb{N}^*.
\]

We compute
\[
4x-n = 4\frac{(n+b)}{4} - n = n+b-n = b .
\]

From the definition

\(F_{x,t}^{(4)}(n)=t^2(4x-n)^2 - 2nxt\)

the condition \(F_{x,t}^{(4)}(n)=0\) gives
\[
t = \frac{2nx}{(4x-n)^2}, 
\]

Since $4x-n = b$, this becomes
\[
t = \frac{2nx}{b^2}.
\]

Substituting $n=4r+1$ and $x=\frac{(4r+1+b)}{4}$ yields
\[
t
= \frac{2(4r+1)\frac{4r+1+b}{4}}{b^2}
= \frac{(4r+1)(4r+1+b)}{2b^2}.
\]

Assume now that $b \mid (4r+1)$. Then there exists an integer $w \ge 1$ such that
\[
4r+1 = bw.
\]

Hence
\[
4r+1+b = bw + b = b(w+1),
\]
and therefore
\[
(4r+1)(4r+1+b) = b^2 w(w+1).
\]

Substituting into the expression of $t$, we obtain
\[
t = \frac{b^2 w(w+1)}{2b^2}
  = \frac{w(w+1)}{2}.
\]

Since $w$ and $w+1$ are consecutive integers, one of them is even.
Thus $w(w+1)$ is divisible by $2$, and consequently
\[
t \in \mathbb{N}^*.
\]

Therefore, for every $b \equiv 3 \pmod{4}$, the condition
\[
b \mid (4r+1)
\]
defines an infinite arithmetic progression in $r$, and hence in $n$.
This provides infinitely many explicit subfamilies in the class
$n \equiv 1 \pmod{4}$ yielding solutions
\[
y = z = t(4x-n)=tb
\]
that satisfy the Erd\H{o}s--Straus equation.
\end{proof}

\noindent\textbf{Example}  
Let $b=3$, which satisfies $3 \equiv 3 \pmod{4}$.

The condition $3 \mid (4r+1)$ is equivalent to
\[
4r+1 \equiv 0 \pmod{3}.
\]

Since $4 \equiv 1 \pmod{3}$, this becomes
\[
r+1 \equiv 0 \pmod{3},
\]
hence
\[
r \equiv 2 \pmod{3}.
\]

We may write $r = 3a+2$. Then
\[
n = 4r+1 = 4(3a+2)+1 = 12a+9.
\]

Thus all integers of the form
\[
n = 12a+9
\]
form an infinite subfamily of integers congruent to $1 \pmod{4}$
for which the above construction produces an integer $t$ and hence a solution arising from $F=0$.

\subsection*{An example of application to a residue class modulo 840}

Mordell proved that for every integer $n$ not congruent to $1, 11^2, 13^2, 17^2, 19^2, 23^2$ modulo $840$, the equation
\[
\frac{4}{n} = \frac{1}{x} + \frac{1}{y} + \frac{1}{z}
\]
admits a solution. Thus, numbers $n \equiv 1 \pmod{840}$ belong to the cases not covered by Mordell's result.

In this section, we show that our construction, based on the existence of a divisor $b \equiv 3 \pmod{4}$ of $n$, provides infinitely many symmetric solutions for some of these numbers, namely those of the form $n = 840k + 1$.

Consider for instance $b = 11$, which satisfies $11 \equiv 3 \pmod{4}$. We look for integers $k$ such that $11$ divides $n = 840k + 1$.

The condition $11 \mid (840k + 1)$ is equivalent to
\[
840k + 1 \equiv 0 \pmod{11} \quad \Longleftrightarrow \quad 840k \equiv -1 \equiv 10 \pmod{11}.
\]

Since $840 \equiv 4 \pmod{11}$, this becomes
\[
4k \equiv 10 \pmod{11}.
\]

We look for a multiple of $4$ that is congruent to $10$ modulo $11$. One finds that $4 \times 8 = 32 \equiv 10 \pmod{11}$. Hence
\[
4k \equiv 4 \times 8 \pmod{11}.
\]

Since $\gcd(4,11)=1$, Gauss's lemma implies that
\[
k \equiv 8 \pmod{11}.
\]

Thus, all integers $k$ of the form $k = 8 + 11m$ (with $m \in \mathbb{N}$) satisfy the condition. This yields infinitely many values of $k$.

Example : For $m = 0$, we have $k = 8$ and $n = 840 \times 8 + 1 = 6721$. One checks that $6721 = 11 \times 611$. Since $n$ is divisible by $b = 11$ and $11 \equiv 3 \pmod{4}$, our construction  applies. We obtain
\[
x = \frac{n + b}{4} = \frac{6721 + 11}{4} = 1683,
\]
\[
t = \frac{n(n + b)}{2b^2} = \frac{6721 \times 6732}{2 \times 121} = 186\,966,
\]
and therefore
\[
y = z = t \times b = 186\,966 \times 11 = 2\,056\,626.
\]

A direct verification shows that these values satisfy
\[
\frac{4}{6721} = \frac{1}{1683} + \frac{1}{2\,056\,626} + \frac{1}{2\,056\,626}.
\].

The same approach can be applied to the five other residue classes in Mordell's theorem — namely, the integers \(n \equiv  121, 169, 289, 361, 529 \pmod{840}\) not covered by his result — by explicitly choosing an integer \(b \equiv 3 \pmod{4}\) dividing \(n\).  
In what follows, we prove a stronger result: we show that the natural density of integers \(n \equiv 1 \pmod{4}\) admitting a divisor \(b \equiv 3 \pmod{4}\) is equal to \(1\). 

\section*{Natural density of integers \(n \equiv 1 \pmod{4}\) admitting a prime divisor \( \equiv 3 \pmod{4}\)}

\subsection*{1. Definition of the sets}

Let \(\mathbb{P}\) denote the set of prime numbers. We define

\[
\mathcal{A}
=
\{ n \ge 1 : n \equiv 1 \pmod{4},\ \exists p \in \mathbb{P},\ p \equiv 3 \pmod{4},\ p \mid n \},
\]

and

\[
\mathcal{B}
=
\{ n \ge 1 : n \equiv 1 \pmod{4},\ \forall p \in \mathbb{P},\ (p \mid n \Rightarrow p \equiv 1 \pmod{4}) \}.
\]

The elements of \(\mathcal{B}\) are exactly the integers not captured by our construction.

We aim to show:

\[
\frac{|\mathcal{B} \cap [1,x]|}{x} \longrightarrow 0
\quad \text{as } x \to +\infty.
\]

\subsection*{2. Multiplicative structure}

An integer n  belongs to \(\mathcal{B}\) if and only if

\[
n = \prod_{p \equiv 1 \pmod{4}} p^{\alpha_p},
\]
where \(\alpha_p \ge 0\) and almost all \(\alpha_p\) are zero.

Thus \(\mathcal{B}\) is a multiplicative set.

We introduce its Dirichlet series, denoted by \(D(s)\):

\[
D(s) = \sum_{n \in \mathcal{B}} \frac{1}{n^s}.
\]

For real \(s > 1\), the series converges absolutely. By multiplicativity, we obtain the Euler product:

\[
D(s) = \prod_{p \equiv 1 \pmod{4}} \left(1 - \frac{1}{p^s}\right)^{-1}.
\]

We now compare with the Riemann zeta function:

\[
\zeta(s) = \prod_{p} \left(1 - \frac{1}{p^s}\right)^{-1}.
\]

To relate \(D(s)\) to \(\zeta(s)\), we separate the prime \(p = 2\) (which does not appear in \(\mathcal{B}\)) and the primes congruent to \(3\) modulo \(4\):

\[
\zeta(s) = \left(1 - \frac{1}{2^s}\right)^{-1}
          \prod_{p \equiv 1 \pmod{4}} \left(1 - \frac{1}{p^s}\right)^{-1}
          \prod_{p \equiv 3 \pmod{4}} \left(1 - \frac{1}{p^s}\right)^{-1}.
\]

It follows that

\[
\prod_{p \equiv 1 \pmod{4}} \left(1 - \frac{1}{p^s}\right)^{-1}
= \zeta(s) \left(1 - \frac{1}{2^s}\right)
  \prod_{p \equiv 3 \pmod{4}} \left(1 - \frac{1}{p^s}\right).
\]

Therefore,

\[
D(s) = \zeta(s) \left(1 - \frac{1}{2^s}\right)
        \prod_{p \equiv 3 \pmod{4}} \left(1 - \frac{1}{p^s}\right).
\]

\subsection*{3. Behavior near \(s=1\)}

It is well known that, as \(s \to 1^+\) (real),

\[
\zeta(s) \sim \frac{1}{s-1}.
\]

Consider the product:

\[
P(s) = \prod_{p \equiv 3 \pmod{4}} \left(1 - \frac{1}{p^s}\right).
\]

Using the logarithmic approximation:

\[
\log P(s) = \sum_{p \equiv 3 \pmod{4}} \log\left(1 - \frac{1}{p^s}\right)
           \sim -\sum_{p \equiv 3 \pmod{4}} \frac{1}{p^s}.
\]

By Dirichlet's theorem on arithmetic progressions,

\[
\sum_{p \equiv 3 \pmod{4}} \frac{1}{p} = +\infty.
\]

Therefore, as \(s \to 1^+\),

\[
\sum_{p \equiv 3 \pmod{4}} \frac{1}{p^s} \longrightarrow +\infty,
\]

and consequently

\[
P(s) \longrightarrow 0.
\]

Moreover, the factor \(\left(1 - \frac{1}{2^s}\right)\) tends to the nonzero constant \(1 - \frac{1}{2} = \frac{1}{2}\) as \(s \to 1^+\). Hence it does not affect the vanishing behavior of \(D(s)\) relative to \(\zeta(s)\).

It follows that

\[
D(s) = o\!\left(\frac{1}{s-1}\right).
\]

\subsection*{4. Tauberian consequence}

Let

\[
B(x)=|\mathcal{B}\cap[1,x]|.
\]

The Dirichlet series

\[
D(s)=\sum_{n\in\mathcal{B}}\frac1{n^s}
\]

has nonnegative coefficients and satisfies

\[
D(s)=o((s-1)^{-1})
\quad \text{as } s\to1^+.
\]

A standard Tauberian theorem for Dirichlet series with nonnegative 
coefficients  then implies

\[
B(x)=o(x).
\]

Hence the set \(\mathcal{B}\) has natural density zero.

Thus,

\[
|\mathcal{B} \cap [1,x]| = o(x).
\]

Hence

\[
\boxed{\operatorname{dens}(\mathcal{B}) = 0.}
\]

\subsection*{5. Conclusion}

Consequently,
\[
\frac{|\mathcal{A} \cap [1,x]|}{x} = 1 - o(1).
\]

In other words, the proportion of integers \(n \equiv 1 \pmod{4}\) possessing at least one prime divisor \(p \equiv 3 \pmod{4}\) tends to \(1\) as \(x \to +\infty\). 

This density result can be intuitively understood: as an integer \(n \equiv 1 \pmod{4}\) grows larger, it becomes increasingly unlikely that it avoids having at least one prime divisor congruent to \(3 \pmod{4}\), and hence almost all such integers are captured by our construction.
\vspace{0.3cm}

We have exhibited explicit solutions for all integers \(n \not\equiv 1 \pmod{4}\), 
and for the class \(n \equiv 1 \pmod{4}\), we have proved that our construction 
applies to almost all such integers, the exceptions having natural density zero. 
In particular, this provides infinitely many new explicit families of symmetric 
solutions for numbers not covered by Mordell's theorem, and establishes that 
the Erdős--Straus conjecture holds for a proportion of integers tending to \(1\).

\end{document}